\documentclass[reqno,10pt]{amsart}
\usepackage{latexsym}
\usepackage{euscript}
\usepackage{amssymb}
\usepackage{amsmath}
\usepackage{float}      % for tabular
\usepackage{array}      % for picture environment 
\usepackage{graphicx}
\usepackage{epic}            %for multiputlist
\usepackage{mathdots}    % for stuff like \iddotsarg
\usepackage{enumitem}   % for example \begin{enumerate}[label=\upshape{(\alph*)}]
\usepackage{color}
\usepackage{mathdots}  %% for stuff like  \iddots

\newtheorem{Thm}[equation]{Theorem}
\newtheorem{Lem}[equation]{Lemma}

\newtheorem{Prop}[equation]{Proposition}

\newtheorem{Obs}[equation]{Observation}

\theoremstyle{remark}

\newtheorem{Example}[equation]{Example}

\numberwithin{equation}{section}

\newcommand{\R}{\mathbf R}           % Use for real numbers.
\newcommand{\C}{\mathbf C}           % Use for complex numbers.
\newcommand{\N}{\mathbf N}           % Use for positive integers.
\newcommand{\Ad}{\operatorname{Ad}}

\newcommand{\rank}{\operatorname{rank}}
\newcommand{\spn}{\operatorname{span}}

\newcommand{\supp}{\operatorname{supp}}
\newcommand{\ann}{\operatorname{ann}}

\newcommand{\f}[1]{{\mathfrak{#1}}}           % Fraktur
\newcommand{\fb}{{\mathfrak b}}
\newcommand{\fg}{{\mathfrak g}}
\newcommand{\fh}{{\mathfrak h}}
\newcommand{\fk}{{\mathfrak k}}

\newcommand{\fn}{{\mathfrak n}}
\newcommand{\fnbar}{{\overline{\mathfrak n}}}
\newcommand{\fp}{{\mathfrak p}}

\newcommand{\ga}{\alpha}
\newcommand{\gb}{\beta}

\newcommand{\gl}{\lambda}

\newcommand{\gD}{\Delta}

\renewcommand{\gg}{\gamma}

\newcommand{\gs}{\sigma}

\newcommand{\gt}{\theta}

\newcommand{\eps}{\varepsilon}

\newcommand{\Cal}{\mathcal}
\newcommand{\Eul}{\EuScript}

\newcommand{\fQ}{{\Eul Q}}
\newcommand{\fB}{{\mathfrak B}}
\newcommand{\eC}{{\EuScript{C}}}

\newcommand{\cD}{{\mathcal D}}

\newcommand{\cF}{\mathcal F}

\newcommand{\cM}{\mathcal M}
\newcommand{\cN}{{\mathcal N}}
\newcommand{\cO}{{\mathcal O}}

\newcommand{\cU}{\mathcal U}
\newcommand{\cW}{\mathcal W}
\newcommand{\Cl}{\Cal C\hskip-1pt\ell}

\newcommand{\Tbar}{\bar{T}}

\renewcommand{\bar}[1]{\overline{#1}}

\newcommand{\IP}[2]{\langle#1\,, #2\rangle}     %inner product.
\newcommand{\T}{{\mathbf {T}}}

\begin{document}

\parskip=4pt 
\baselineskip=14pt

\title[Characteristic cycles]{Characteristic cycles of highest weight Harish-Chandra modules}

\author{R. Zierau}

\address{Oklahoma State University \\
   Mathematics Department \\
   Stillwater, Oklahoma 74078}
      \email{roger.zierau@okstate.edu}

%\subjclass[2010]{22E46}   

\begin{abstract} Characteristic cycles and leading term cycles of irreducible highest weight Harish-Chandra modules of regular integral infinitesimal character are determined.  In the simply laced cases they are irreducible, but in the nonsimply laced cases they are more complicated.
\end{abstract}

%\date{\today}

\maketitle

%{\small\tableofcontents}

%%%%%%%%%%

%\setlength{\parindent}{0cm}

\section{Introduction}
Characteristic cycles play an important role in the representation theory of real reductive Lie groups.  For example they arise in the study of primitive ideals, the associated variety and character theory (e.g., \cite{BorhoBrylinski85}, \cite{Rossmann91} and \cite{SchmidVilonen00}).  In this article the characteristic cycles of highest weight Harish-Chandra modules having regular integral  infinitesimal character  are given explicitly.

The Harish-Chandra module of each irreducible representation $X$ of a reductive group $G_\R$  has  a `support' associated to it.  This support is the support of the $\cD$-module on the flag variety $\fB$ given by the localization of $X$ (\cite{BeilinsonBernstein81}, \cite{HechtMilicicSchmidWolf87}); it is the closure of a $K$-orbit in the flag variety $\fB$, where $K$ is the complexification of a maximal compact subgroup of $G_\R$.  The support may also be described in terms of the Langlands classification.  The characteristic cycle of $X$ is an element of  top degree in Borel-Moore homology of the conormal variety for the action of $K$ on $\fB$:
\begin{equation}\label{eqn:CC}
CC(X)=\sum m_\fQ[\bar{T_\fQ^*\fB}],
\end{equation}
with the sum over $K$-orbits $\fQ$ in $\fB$.   The `multiplicities' $m_\fQ$ are nonnegative integers.  The characteristic cycle of $X$ may be defined in terms of filtrations of the corresponding $\cD$-module (\cite{BorhoBrylinski85}).  Through the Riemann-Hilbert correspondence this is equivalent to the characteristic cycle of an $IC$ sheaf associated to $\supp(X)$ (\cite{Kashiwara85}).  The characteristic cycle depends  on the singularities of $\supp(X)\subset\fB$.

Boe and Fu (\cite{BoeFu97}) have computed  characteristic cycles of $IC$ sheaves for Schubert varieties in certain generalized grassmannian flag varieties for classical groups.  They  carry  out explicit geometric computations and use the combinatorics of the Kazhdan-Lusztig polynomials to obtain characteristic cycles.  They find that in the simply laced cases the characteristic cycles are irreducible (i.e., just one term appears in (\ref{eqn:CC})), otherwise they give a procedure to find the characteristic cycles.  It turns out that this work is related to computing characteristic cycles of highest weight Harish-Chandra modules (in a somewhat indirect way).

In the case of $G_\R=Sp(2n,\R)$, the most difficult case, the computation of characteristic cycles of highest weight Harish-Chandra modules is contained in \cite{BarchiniZierau17}.  The methods are quite different from those of \cite{BoeFu97} and the results are stated in a much different way (and it is not clear how to match up the statements).

\pagebreak
The first theorem in this article is the following
\begin{Thm}
If $G_\R$ is a simple group of hermitian type with simply laced Lie algebra, then for any irreducible highest weight Harish-Chandra module $X$ of regular integral infinitesimal character 
$$
CC(X)=[\bar{T_\fQ^*\fB}],
$$
where $\bar{\fQ}=\supp(X)$.
\end{Thm}
In the classical cases this follows from \cite{BoeFu97}; we explain this in Section \ref{sec:sl}.  The proofs in the exceptional cases are contained in Section \ref{sec:exceptional}.  In the nonsimply laced cases the characteristic cycle can contain more that one term.  The case of $G_\R=Sp(2n,\R)$, as mentioned above, is contained in \cite{BarchiniZierau17}.  Here we state the result as a very simple algorithm.  In Appendix \ref{appendix:B} we use this algorithm to compute the fraction of highest weight Harish-Chandra modules (of a given regular integral infinitesimal character) having irreducible characteristic cycle.  We see that this fraction is very small as the rank increases.  For $SO_e(2,2n-1)$, we find the characteristic cycles and see that they are either irreducible or a sum of two terms, each with multiplicity one (Section \ref{sec:so}).

The methods used here are to apply a number of well-known general facts about characteristic cycles, Harish-Chandra cells and Weyl group representations, along with explicit computations.  We also compute certain singularities that occur in $\bar{\fQ}$ for several $K$-orbits $\fQ$ in $\fB$.

%%%%%%%%%%%%%%%%%%%%
\section{Preliminaries}
The arguments in this article depend on a number of known  general facts about highest weight Harish-Chandra modules and characteristic cycles.  In this section we gather together the information needed.

A highest weight Harish-Chandra module is an irreducible Harish-Chandra module having a vector annihilated by all root vectors  $X_\ga$ for $\ga$ in some positive system of roots (for some Cartan subalgebra).  The simple groups for which infinite dimensional highest weight Harish-Chandra modules occur are those groups of hermitian type.

Let $G_\R$ be one of the simple groups of hermitian type.  Let $\fg$ be the complexification of its Lie algebra and let $\gt$ be a Cartan involution.  Complexifying  $\gt$ gives the complexified  Cartan decomposition  $\fg=\fk\oplus\fp$.  We let $K$ be the complexification of the corresponding maximal compact subgroup (so $\rm{Lie}(K)=\fk$).  It is a fact that for hermitian type groups $\rank(\fk)=\rank(\fg)$, so we may fix a Cartan subalgebra $\fh$ of $\fg$ that is contained in $\fk$.  The set of roots of $\fh$ in $\fg$ (resp., $\fk$) is denoted by $\gD$ (resp., $\gD_c$).  The Weyl group of $\fg$ (resp., $\fk$) is denoted by $W$ (resp., $W_c$).   The flag variety of $\fg$ is denoted by $\fB$. The nilpotent cone in $\fg$ is denoted by $\cN$.  We denote by $G$ the adjoint group for $\fg$.  Given a positive system $\gD^+$ in $\gD$ we use the usual notation of $\rho$ for half the sum of the positive roots.

An equivalent condition for a simple group $G_\R$ to be of hermitian type is that $\fp$ decomposes into the direct sum of two irreducible $\Ad(K)$-subrepresentations. This is written as $\fp=\fp_+\oplus\fp_-$.  It is well-known that a highest weight Harish-Chandra module is highest weight with respect to a positive system (for Cartan subalgebra $\fh$) of the form $\gD_c^+\cup\gD(\fp_\pm)$, where $\gD(\fp_\pm):=\{\ga\in\gD:\fg^{(\ga)}\subset\fp_\pm\}$ and $\gD_c^+ $ is any positive system for $\gD_c$.  Since there is no loss of generality, we assume our highest weight Harish-Chandra modules have highest weight with respect to $\gD^+=\gD_c^+\cup\gD(\fp_+)$.  Therefore, each highest weight Harish-Chandra module is the irreducible quotient of a generalized Verma module 
$$\cU(\fg)\underset{\cU(\fk+\fp_+)}{\otimes} F_\gl,
$$
with $F_\gl$ an irreducible representation of $\fk$ of highest weight $\gl$ (with respect to $\gD_c^+$).  Those of infinitesimal character $\rho$ may be written as  $L_w$, with $\gl=-w\rho-\rho$, for $w$ in 
$$\cW:=\{w\in W:-w\rho\text{ is $\gD_c^+$-dominant}\}.
$$

It is well-known that the associated variety of any highest weight Harish-Chandra module is contained in $\fp_+$ (\cite{NishiyamaOchiaiTaniguchi01}).  The $K$-orbits in $\fp_+$ are $\cO_j$, $j=0,1,2,\dots,r:=\rank_\R(G_\R)$:
\begin{equation*}
\cO_0=\{0\},\text{ and } \cO_j=K\cdot(X_{\gg_1}+\dots +X_{\gg_j}), j=1,2,\dots,r,
\end{equation*}
where $\{\gg_1,\dots,\gg_r\}$ is a maximal set of strongly orthogonal roots in $\fp_+$ and $X_{\gg}\in\fg$ is a $\gg$-root vector.  For example, when $G_\R=Sp(2n,\R)$, $\fp_+\simeq \operatorname{Sym}(n,\C)$ and $\cO_j$ is the  set of rank $j$ symmetric matrices.  In general these orbits have the following inclusion relations
\begin{equation}\label{eqn:inclusions}
\cO_0\subset\bar{\cO_1}\subset \bar{\cO_2}  \subset \cdots \subset\bar{\cO_r}=\fp_+.
\end{equation}
It follows that each highest weight Harish-Chandra module has associated variety equal to one of the $\bar{\cO_j}$.

Now we state some general facts about characteristic cycles.  By the translation principle, in order to determine the characteristic cycles of highest weight Harish-Chandra modules of regular integral infinitesimal character, it suffices to consider only those of infinitesimal character $\rho$, that is, the $L_w, w\in\cW$.  Thus, we work with these representations.

The first observation is that the $L_w, w\in \cW$, may be studied in either the category $\cM_\rho(\fg,K)$ of Harish-Chandra modules of infinitesimal character $\rho$ or in $\cM_\rho(\fg,B)$, a category of highest weight $(\fg,B)$-modules of infinitesimal character $\rho$, where $B$ is the Borel subgroup corresponding to the positive system $\gD^+=\gD_c^+\cup\gD(\fp_+)$.  Characteristic cycles in both settings are studied simultaneously in \cite{BorhoBrylinski85}.  For Harish-Chandra modules the characteristic cycles are in terms of the conormal variety for the $K$-action on $\fB$, as stated above.  In $\cM_\rho(\fg,B)$ the support of the irreducible module $L_w$ is the Schubert variety $Z_w=\bar{Bw\cdot\fb}$ and the characteristic cycle is of the form
$$
CC(L_w)=\sum_{y\in W}m_{y,w}[\bar{T_{B_y}^*\fB}],
$$
where $B_y=By\cdot\fb\subset \fB$.  Let $\mu:T^*\fB\to \cN$ be the moment map.  If $\fQ$ is a $K$-orbit in $\fB$, then $\mu(\bar{T_\fQ^*\fB})$ is the closure of a $K$-orbit in $\cN\cap\fp$.  If we set $\fn^w:=\Ad(w)(\fn)$, then $\mu(\bar{T_{B_w}^*\fB})=B\cdot(\fn\cap\fn^w)$, an orbital variety.    These are general statements.  For highest weight Harish-Chandra modules one sees that the geometric objects arising in the two categories coincide, as the following shows.  Recall that $B$ is a Borel subgroup corresponding to $\gD^+=\gD_c^+\cup\gD(\fp_+)$, for some fixed (choice of) positive system in $\gD_c^+$.
\begin{Lem}[\cite{BarchiniZierau17}]  The following are equivalent.
\begin{enumerate}[label=\upshape{(\arabic*)}]
\item $Z_w=\bar{B_w}$ is the closure of a $K$-orbit in $\fB$.
\item $B\cdot(\fn\cap\fn^w)$ is the closure of a $K$-orbit in $\fp_+$.
\item $-w\rho-\rho$ is $\gD_c^+$-dominant.
\end{enumerate}
\end{Lem}
Therefore we may give our statements and computations for highest weight Harish-Chandra modules in terms of either $K$-orbits in $\fB$ and in $\cN\cap\fp$ \emph{or} in terms of Schubert varieties and orbital varieties.

We now review  facts we need about characteristic cycles.  A reference for much of this is \cite{BorhoBrylinski85}.
\begin{Thm}\label{thm:cc} If $X$ is an irreducible Harish-Chandra module with support $\bar{\fQ}$, then 
$$
CC(X)=\sum m_{\fQ'} [\bar{T_{\fQ'}^*\fB}],
$$
with the multiplicities satisfying
\begin{enumerate}[label=\upshape{(\arabic*)}]
\item $m_\fQ=1$ and 
\item $m_{\fQ'}\neq 0$ only when $\fQ'\subset\bar{\fQ}$.
\end{enumerate}
\end{Thm}
\noindent A similar statement holds for irreducibles in $\cM_\rho(\fg,B)$.

The associated variety $A(X)$ of $X$ (\cite{Vogan91}) is given in terms of the characteristic cycle by the following formulas.  Let $\mu:T^*\fB\to \cN$ be the moment map.  Then
\begin{subequations}\label{eqn:av}
\begin{equation}
AV(X)=\bigcup_{m_{\fQ'}\neq 0} \mu(\bar{T_{\fQ'}^*\fB})
\end{equation}
and
\begin{equation}
AV(L_w)=\bigcup_{m_{y,w}\neq 0} \mu(\bar{T_{B_y}^*\fB}).
\end{equation}
\end{subequations}
This is an equidimensional affine variety  (in $\cN\cap\fp$).  The \emph{leading term cycle} (\cite{Trapa07}) is 
$$
LTC(X)=\sum_{\fQ'\in LT} m_{\fQ'}[\bar{T_{\fQ'}^*\fB}],
$$
with $LT:=\{\fQ' : \dim(\mu(\bar{T_{\fQ'}^*\fB}))=\dim(AV(X))\}$.  
\begin{Obs} \label{obs:av} If $AV(X)$ is known, then (\ref{eqn:av}) excludes $\fQ'$ from occuring in $CC(X)$ (with $m_{\fQ'}\neq 0$) when $\dim(\mu(\bar{T_{\fQ'}^*\fB}))>\dim(AV(X)).$
\end{Obs}

Another way to exclude certain potential terms from ocuring in the characteristic cycle is by the $\tau$-invarant, a subset of the set of simple roots $\Pi$.  This is best expressed in terms of $\cM_\rho(\fg,B)$.  The $\tau$-invariant of $w\in W$ is defined as
$$
\tau(w):=\{\ga\in \Pi: w(\ga)<0\}.
$$
\begin{Prop}[\cite{BorhoBrylinski85}]  \label{prop:tau}
Write $CC(L_w)=\sum_{y\in W} m_{y,w}[\bar{T_{B_y}^*\fB}]$.  Then $m_{y,w}\neq 0$ only when $\tau(w)\subset\tau(y)$.
\end{Prop}

Another statement involving the $\tau$-invariant will be used.  This is stated in terms of the $\T_{\ga\gb}$-operators.  Let $\ga$ and $\gb$ be two simple roots of the same length so that $\IP{\ga}{\gb}\neq 0$.  One says that $w$ is in the domain of $\T_{\ga\gb}$ if $\ga\notin\tau(w)$ and $\gb\in\tau(w)$.  When $w$ is in the domain of $\T_{\ga\gb}$, then 
$$
\T_{\ga\gb}(w)=\begin{cases}  ws_\ga,    &\gb\notin \tau(ws_\ga),  \\
                             ws_\gb,    &\ga\in\tau(ws_\gb).
   \end{cases}
$$
(When $\ga$ and $\gb$ have unequal lengths the situation is slightly more complicated;  this case is omitted here since we do not need it.)
\begin{Prop}[\cite{McGovern00}]\label{prop:mcgovern}
Suppose $\ga,\gb\in \Pi$ have the same length and $\IP{\ga}{\gb}\neq 0$.  If $w,y$ are both in the domain of $\T_{\ga\gb}$, then $m_{y,w}=m_{\T_{\ga\gb}(y),\T_{\ga\gb}(w)}$.
\end{Prop}
\noindent This will give us good information about several $CC(L_{\T_{\ga\gb}(w)})$ once we know $CC(L_w)$.

The notion of Harish-Chandra cell naturally arises.  The (finite) set of irreducible Harish-Chandra modules of infinitesimal character $\rho$ is partitioned into Harish-Chandra cells by an equivalence relation $\sim$.  We say that $X\sim Y$ when $X\prec Y$ and $Y\prec X$, for the partial order $\prec$ generated by $X\prec Y$ when $X$ occurs as a subquotient in $Y\otimes F$, for some finite dimensional representation $F$ that is isomorphic to a subrepresentation of the tensor algebra of $\fg$.  It easily follows that two representations in a Harish-Chandra cell have the same associated variety.

If $\eC$ is a Harish-Chandra cell, then $V_\eC:=\spn\{X:X\in \eC\}$ may be identified with a subquotient of the coherent continuation representation of $W$ acting on the Grothendieck group of $\cM_\rho(\fg,K)$.  This `cell representation' need not be irreducible, but it always contains a `special' representation.  This constituent may be described as follows.  The associated variety of the annihilator of some $X\in\eC$ is the closure of a single nilpotent $G$-orbit $\cO^\C\subset \cN$ (\cite{Vogan91}).  It is a fact that $AV(\ann(X))$ is the same orbit closure for all $X\in\eC$  and if $\bar{\cO}$ is an irreducible component of $AV(X)$, then $G\cdot\cO=\cO^\C$.   Then $V_\eC$ contains $\pi(\cO^\C,1)$, the irreducible representation of $W$ corresponding to $\cO^\C$ under the Springer correspondence; this representation is special in the sense of Lusztig.  There are tables giving the Springer correspondence and indicating which representations are special (e.g., \cite{Carter85}).

%\newpage
%%%%%%%%%%%%%%%%%%%%
\section{Classical simply laced cases}\label{sec:sl}
Let $G_\R$ be one of the simple classical groups of hermitian type for which $\fg$ is simply laced.  These are
\begin{equation}\label{eqn:sl-classical}
SU(p,q), SO_e(2,2n-2)\text{ and } SO^*(2n).
\end{equation}
Consider the Schubert varieties in the generalized flag variety $\cF:=G/KP_+$, that is the closures of the $B$-orbits in $\cF$.  As mentioned in the introduction, it is shown in \cite{BoeFu97} that the characteristic cycles of the $IC$ sheaves for these Schubert varietes are irreducible.  We may use this to show that all characteristic cycles of highest weight Harish-Chandra modules for the groups in (\ref{eqn:sl-classical}) are irreducible.  

To accomplish this we need several facts.  Let $\pi:\fB\to\cF$ be the natural fiber bundle.
\begin{Lem}
Let $w\in W$.  Then $w\in\cW$ if and only if $w^{-1}$ is maximal in its $W_c$-coset. \end{Lem}
\begin{proof} We first show that $w$ is maximal in its $W_c$-coset if and only if $w\gD_c^+\subset -\gD^+$.   If $w$ is maximal in its coset, then $\ell(ws_\ga)\leq\ell(w)$ for all simple $\ga\in\gD_c^+$.  This implies $w(\gb)<0$, for all $\gb\in\gD_c^+$.  Conversely, if $w\gD_c^+\subset-\gD^+$, then $\ell(w)=\#A + \#\gD_c^+$, , for 
$$
A:=\{\gb\in\gD^+\smallsetminus\gD_c^+: w(\gb)<0\}.
$$
  Now let $y\in W_c$.  Since $K$ normalizes $\fp_+$, $y(\gD^+\smallsetminus\gD_c^+)=\gD^+\smallsetminus\gD_c^+$, so $y^{-1}A=\{\gg\in \gD^+\smallsetminus\gD_c^+ : w(\gg)<0\}$.  Therefore, $\ell(wy)\leq \#(y^{-1}A)+\#\gD_c^+ = \#A + \#\gD_c^+ = \ell(w).$  So $w$ is maximal in its coset.

For $\ga\in\gD_c^{+}$,  $w\ga<0$ if and only if $(-\rho,w\ga)>0$.  Therefore, $w$ is of maximal length in its coset if and only if  $-w^{-1}\rho$ is $\gD_c^+$-dominant.  The lemma follows.
\end{proof}

We make use of several general facts.
\begin{enumerate}
\item If $y\in W$ is maximal in its $W_c$ coset, then $\pi|_{Z_y}:Z_y\to \pi(Z_y)$ is a fibration with smooth fiber.  The  $IC$ sheaf for $Z_y$ has  irreducible characteristic cycle if and only if the same holds for the $IC$ sheaf for $\pi(Z_y)$.
\item The characteristic cycle of the $IC$ sheaf for $Z_y$ is $CC(L_y)$, for any $y\in W$ (\cite{Kashiwara85}).
\item For any $w\in W$, $CC(L_w)$ is irreducible if and only if $CC(L_{w^{-1}})$ is irreducible (\cite[\S6.9]{BorhoBrylinski85}).
\end{enumerate}

\begin{Thm}
If $G_\R$ is one of the groups in (\ref{eqn:sl-classical}), then $CC(L_w)=[\bar{T_{B_w}^*\fB}]$, for each $w\in\cW$.
\end{Thm}
\begin{proof}  Let $w\in\cW$. It follows from \cite{BoeFu97} that the characteristic cycle of the $IC$ sheaf for $\pi(Z_{w^{-1}})$ is irreducible.  By points (1) and (2), $CC(L_{w^{-1}})$ is irreducible.  Now, by (3), $CC(L_w)$ is irreducible.
\end{proof}

%%%%%%%%%%%%%%%%%%%%%%%%%%%%%%%%%%%%%%%%%%%%%%%%%%%%%%%%%%%%%%%%%%%%%%%%%%%%%%%%%%%%%%%%%%%%%%5
\section{$E_6$ and $E_7$}\label{sec:exceptional}

As in the classical simply laced cases, the characteristic cycles of highest weight Harish-Chandra modules are irreducible in the two exceptional cases.  This is proved by using Proposition \ref{prop:mcgovern} that  relates the $\T_{\ga\gb}$-operators to the characteristic cycle.

\begin{Thm}\label{thm:cc-exeptional}  If $G_\R$ is $E_{6(-14)}$ or $E_{7(-25)}$, then $CC(L_w)=[\bar{T_{B_w}^*\fB}]$, for all $w\in\cW$.
\end{Thm}

First consider the $E_6$ case.  Then $K=\C^\times \times Spin(10,\C)$ and the number of highest weight Harish-Chandra modules is $27$.  The simple roots are given the usual numbering and $\ga_1$ is the noncompact simple root.  The elements of $\cW$ are listed as $w_1,\dots,w_{27}$ in Table \ref{tab:w-e6} in the appendix.  

In Table 1 we list the corresponding $\tau$-invariants (by giving sets of indices of simple roots), the dimension of the support, the associated variety and what we call the `possible' characteristic cycles (expressed by giving a set of indices $k$ for which $\bar{T_{B_{w_k}}^*\fB}$ could occur in $CC(L_{w_i})$).  By `possible' we mean those conormal bundle closures that cannot be excluded from the characteristic cycle by considering columns 2-4 and Theorem \ref{thm:cc}, Observation \ref{obs:av} (along with (\ref{eqn:inclusions})) and Proposition \ref{prop:tau}.

\begin{table}[H]\label{table:e6-data}
\scalebox{.8}{\mbox{\hspace{0cm}
$
\setlength{\extrarowheight}{3pt}
\begin{array}{rlccc|rlccc}
i & \quad \tau  (w_i)       &  \dim(Z_{w_i})  & AV(L_{w_i}) & \text{Pos. CC} &
i\;\, & \quad \tau  (w_i)       &  \dim(Z_{w_i})  & AV(L_{w_i}) & \text{Pos. CC} \\
\hline  
1     & \{2, 3, 4, 5, 6 \}  &     20   & \cO_2                 & \{1\}   &
15     & \{1, 2, 4, 5, 6 \}  &     28   & \cO_2            & \{2, 15\}   \\
 2     & \{1, 2, 4, 5, 6 \}  &     21   & \cO_2                & \{2\}   &
16     & \{1, 2, 3, 4, 6 \}  &     29   & \cO_2            & \{5, 16\}   \\
 3     & \{1, 2, 3, 5, 6 \}  &     22   & \cO_2                & \{3\}   &
17        & \{2, 3, 5, 6 \}  &     29   & \cO_2          & \{1, 3, 10, 17\}   \\
 4        & \{1, 3, 4, 6 \}  &     23   & \cO_2                 & \{4\}   &
18     & \{1, 2, 3, 5, 6 \}  &     30   & \cO_2        & \{3, 10, 18\}   \\
 5     & \{1, 2, 3, 4, 6 \}  &     24   & \cO_2               & \{5\}   &
19        & \{3, 4, 5, 6 \}  &     30   & \cO_2         & \{1, 8, 19\}   \\
 6        & \{1, 3, 4, 5 \}  &     24   & \cO_2                & \{6\}   &
 20        & \{1, 4, 5, 6 \}  &     31   & \cO_2         & \{2, 8, 15, 20\}  \\
 7        & \{1, 2, 3, 5 \}  &     25   & \cO_2                & \{3, 7\}   &
21     & \{2, 3, 4, 5, 6 \}  &     31   & \cO_1            & \{1, 21\}   \\
 8     & \{1, 3, 4, 5, 6 \}  &     25   & \cO_2                & \{8\}   &
22     & \{1, 2, 4, 5, 6 \}  &     32   & \cO_1               & \{22\}  \\
 9        & \{1, 2, 4, 5 \}  &     26   & \cO_2                & \{2, 9\}   &
23     & \{1, 3, 4, 5, 6 \}  &     32   & \cO_1               & \{23\}    \\
10     & \{1, 2, 3, 5, 6 \}  &     26   & \cO_2              & \{3, 10\}  &
24     & \{1, 2, 3, 5, 6 \}  &     33   & \cO_1               & \{24\}      \\
11        & \{2, 3, 4, 5 \}  &     27   & \cO_2              & \{1, 11\}  &
25     & \{1, 2, 3, 4, 6 \}  &     34   & \cO_1               & \{25\}   \\
12        & \{1, 2, 4, 6 \}  &     27   & \cO_2               & \{2, 5, 12\}   &
26     & \{1, 2, 3, 4, 5 \}  &     35   & \cO_1               & \{26\}   \\
13     & \{1, 2, 3, 4, 5 \}  &     28   & \cO_2               & \{13\}   &
27  & \{1, 2, 3, 4, 5, 6 \}  &     36  & \cO_0              & \{27\} \\
14        & \{2, 3, 4, 6 \}  &     28   & \cO_2             & \{1, 5, 14\}  &&&&& \\
\end{array}
$
}  }
\medskip
\caption{$E_6$ Data}
\end{table}

%\end{minipage}
%\end{sideways}
%\end{table}

The Table 2 gives the $\T_{\ga\gb}$ operators.  Using the usual numbering of simple roots, the pairs of consecutive simple roots are $(\ga_1,\ga_3)$, $(\ga_2,\ga_4)$, $(\ga_3,\ga_4)$, $(\ga_4,\ga_5)$, and $(\ga_5,\ga_6)$.  For each pair, the table lists those $w\in\cW$ that are in the domain of $\T_{\ga_i\ga_j}$ (written as $\T_{ij}$) and $\T_{\ga_i\ga_j}(w)$.  An entry of the table is a natural number $i$ which is short for $w_i$.

\begin{table}[H]
\scalebox{.9}{\mbox{\hspace{0cm}
$
\setlength{\extrarowheight}{3pt}
\begin{array}{cc|cc|cc|cc|cc}
w & \T_{13}(w)       &  w & \T_{34}(w)       &w & \T_{24}(w)       &w & \T_{45}(w)       &w & \T_{56}(w)              \\
\hline  
1     &  2  & 2  &  3   &  4  &  3  & 3   &  4 &  4  & 6       \\
11    &  9  & 9  &  7   &  6  &  7  & 7   &  5 &  5  & 7       \\
14    & 12  & 12 &  10  &  8  & 10  & 10  & 12 & 12  & 9       \\
17    & 15  & 15 &  17  & 19  & 17  & 17  & 14 & 14  & 11      \\
19    & 20  & 20 &  18  & 20  & 18  & 18  & 16 & 16  &  13      \\
21    & 22  & 22 &  24  & 23  & 24  & 24  & 25 & 25  &  26      \\
\end{array}
$
}  }
\medskip
\caption{$\T_{\ga\gb}$'s for $E_6$}
\end{table}

Now we may apply Proposition \ref{prop:mcgovern} to see that all characteristic cycles are irreducible.   For this first note that, by the last column of Table 1, $\Tbar_4:=\bar{T_{B_{w_4}}^*\fB}$ does not occur in any $CC(L_w)$ (other than $CC(L_{w_4})$).  By considering $\T_{24}$, $\Tbar_3$ does not occur in characteristic cycles for $w_7, w_{10}, w_{17}$ and $w_{18}$.  This is because, for example, $w_4, w_6$ are in the domain of $\T_{24}$, so $m_{w_3,w_7}= m_{\T_{24}(w_4),\T_{24}(w_6)}=m_{w_4,w_6}=0$.  By the last column of Table 1, this means that $\Tbar_3$ does not occur  anywhere (other than in $CC(L_{w_3})$).  Now this fact along with the same reasoning applied to $\T_{34}$ tells us that $\Tbar_2$ does not occur anywhere.  Similarly, from $\T_{12}$, we see that $\Tbar_1$ only occurs in $CC(L_{w_1})$.  The last column of Table 1 shows that $\Tbar_7$ only occurs in $CC(L_{w_7})$,  Therefore $\Tbar_5$ occurs only in $CC(L_{w_5})$ (by looking at $\T_{45}$).  One excludes $\Tbar_{10}, \Tbar_8$ and $\Tbar_{15}$ similarly by considering $\T_{45}, \T_{24}$ and $\T_{13}$ (respectively).

For $E_7$, $K=\C^\times\times E_6$ and $\#\cW=56$.  The  elements of $\cW$ are given in  Table 4 in  Appendix A.  The necessary information for $E_7$ is listed in remaining tables in Appendix \ref{appendix:A}.  These tables are easily generated by Maple.  It is easy to check that the same argument for irreducibility  of characteristic cycles as given in the previous paragraph also applies to the $E_7$ case.

%%%%%%%%%%%%%%%%%%%%
\section{$SO_e(2,2n-1)$}\label{sec:so}    We take the simple roots to be 
$$\Pi:=\{\ga_1=\eps_1-\eps_2,  \ga_2=\eps_1-\eps_3,\dots , \ga_{n-1}=\eps_{n-1}-\eps_n,\ga_n=\eps_n\},
$$
with $\ga_1$ noncompact.  Therefore $\gD_c=\{\pm(\eps_i-\eps_j):2\leq i\neq j\leq n\}\cup\{\pm\eps_i:2\leq i\leq n\}$ and $K=\C^\times \times SO(2n-1,\C)$.  The number of highest weight Harish-Chandra modules is $\#\left(W/W_c\right)=2n$.

Since the action of $K$ on $\fp_+$ is by scalars and by the usual action of $SO(2n-1,\C)$ on $\fp_+=\C^{2n-1}$,  the orbits are 
\begin{equation*}
\cO_0=\{0\}, \, \cO_1=\{X\in\fp_+:\IP{X}{X}=0, X\neq 0\}\, \text{ and }
\cO_2=\{X\in\fp_+:\IP{X}{X}\neq 0\},
\end{equation*}
for some $K$-invariant symmetric form $\IP{\,}{\,}$ on $\fp_+$. The partition associated to the complex orbit $\cO_2^\C$ (resp.,  $\cO_1^\C$)  is $[3,1^{2n-2}]$ (resp., $[2^2,1^{2n-2}]$).  Thus, \cite[\S6]{CollingwoodMcgovern93}  easily implies $\cO_2^\C$ is special, but $\cO_1^\C$ is not special.  Alternatively, one may conclude the same by computing the Springer correspondence, as given in \cite{Carter85}.  It is shown in \cite{BarbaschVogan82} that annihilators of irreducible Harish-Chandra modules have associated varieties that are closures of special orbits.  Therefore, associated varieties of the highest weight Harish-Chandra modules are $\bar{\cO_2}$, except for the trivial representation, which has associated variety $\bar{\cO_0}$.  

The highest weight Harish-Chandra modules are $\{L_w,w\in\cW\}$, with $\cW$ described as follows.  For $k=1,2,\dots,n$, let $\tilde{w}_k$ be the permutation $(1,2,3,\dots ,n-k+1)$, written in cycle notation.  Let $w_0$ be the long element of $W$ ($w_0=-Id$ on $\fh^*$) and let $\gs_{\eps_1}$ be the reflection in $\eps_1$.  Set 
\begin{equation*}
w_k^+=w_0\tilde{w}_k\text{ and }\, w_k^-=w_0\gs_{\eps_1}\tilde{w}_k.
\end{equation*}
Then $\cW=\{w_k^+,w_k^-:k=1,\dots,n\}$ and 
\begin{equation*}
-w_k^\pm\rho=(\pm(k-\frac12), n-\frac12, n-\frac32,\dots, \widehat{k-\frac12},\dots,\frac32,\frac12).
\end{equation*}
One easily checks the following.
\begin{Lem}\label{lem:B1}
For $k=1,2,\dots,n$ the following hold.
\begin{enumerate}[label=\upshape{(\alph*)}]
\item $\ell(w_k^+)=(n-1)^2+n+k-1$ and $\ell(w_k^-)=(n-1)^2+n-k$.
\item $\tau(w_k^+)=\tau(w_{k+1}^-)=\Pi\smallsetminus\{\ga_{n-k}\}, k=1,2,\dots,n-1$,\\
      $\tau(w_n^+)=\Pi$ and $\tau(w_1^-)=\Pi\smallsetminus\{\ga_n\}$.
\item $\gD(\fn\cap\fn^{w_k^+})=\{\eps_1-\eps_j:j=2,3,\dots,n-k+1\}, k=1,2,\dots n-1$, \\
      $\gD(\fn\cap\fn^{w_n^+})=\emptyset$, \\
      $\gD(\fn\cap\fn^{w_k^-})=\{\eps_1\}\cup \{\eps_1-\eps_j:j=2,3,\dots,n\}\cup \{\eps_1+\eps_j:j=n-k+2,\dots,n\},\\ \hbox{}\qquad k=2,\dots, n.$ \\
      $\gD(\fn\cap\fn^{w_1^-})=\{\eps_1\}\cup \{\eps_1-\eps_j:j=2,3,\dots,n\}.$ 
           
\item $\mu\left(\bar{ T_{B_{w_n^+}}^*\fB }\right)=\bar{\cO_0}=\{0\},  $  \\
   $\mu\left(\bar{ T_{B_{w_k^+}}^*\fB }\right)=\bar{\cO_1}, k=1,2,\dots, n-1,$  \\
   $\mu\left(\bar{ T_{B_{w_k^-}}^*\fB }\right)=\bar{\cO_2}, k=1,2,\dots, n$.
\end{enumerate}
\end{Lem}
The last statement may be seen from part (c) and noticing when $\fn\cap\fn^{w}$ lies in $\{X\in\fp_+:\IP{X}{X}=0\}$.

By considering the $\tau$-invariant, the closure relations of the Schubert varieties and the moment map images (and applying Theorem \ref{thm:cc}, Observation \ref{obs:av} and Proposition \ref{prop:tau}) we see that 
\begin{align*}
  CC(L_{w_k^+})&=[\bar{T_{B_{w_k^+}}^*\fB}] + m_k[\bar{T_{B_{w_{k+1}^-}}^*\fB}], \;k=1,2,\dots,n-1,  \\
  CC(L_{w_n^+})&=[\bar{T_{B_{w_n^+}}^*\fB}],  \\
\intertext{and}
  CC(L_{w_k^-})&=[\bar{T_{B_{w_k^-}}^*\fB}], \; k=1,2,\dots,n.  \\
\end{align*}
  Note that since $\mu(T_{w_k^+})=\bar{\cO_1}$ and  the associated variety is $\bar{\cO_2}$, we must have $m_k>0$, by (\ref{eqn:av}).

We now show that $m_k=1$, for each $k=1,2,\dots,n$.
\begin{Lem}
Let $k=1,2,\dots,n-2$.
\begin{enumerate}[label=\upshape{(\alph*)}]
\item $\ga_{n-k}\notin \tau(w_k^+)=\tau(w_{k+1}^-)$ and $\ga_{n-k-1}\in \tau(w_k^+)=\tau(w_{k+1}^-)$, therefore $w_k^+,w_{k+1}^-$ are in the domain of $T_{\ga_{n-k},\ga_{n-k-1}}$.
\item $T_{\ga_{n-k},\ga_{n-k-1}}(w_k^+)=w_{k+1}^+$ and $T_{\ga_{n-k},\ga_{n-k-1}}(w_{k+1}^-)=w_{k+2}^-$.
\end{enumerate}
\end{Lem}
\noindent Part (a) is contained in the previous lemma and (b) is an easy calculation.
Applying Proposition \ref{prop:mcgovern} we conclude that $m_1=m_2=\cdots =m_{n-1}$.

\noindent {\bf Claim:} $m_1=1$.
 
The argument for this is to compute a normal slice in $Z_{w_1^+}$ at $B_{w_2^-}$ and compare with a known  characteristic cycle.  This argument was made in \cite[\S4]{BarchiniZierau17} and is  explained in some detail there.  

Let $S=\exp(\fnbar\cap\fnbar^{\,w_2^-})\cdot\fb$, were $\fb$ is the Borel subalgebra corresponding to $\gD^+$.  We need to compute the slice $S\cap Z_{w_1^+}$ and consider the singularity at $0$.

Suppose that $SO(2n+1,\C)$ is defined by the symmetric matrix 
$$
\begin{bmatrix}  &&&& 1 \\  &&& 1 & \\  
            && \iddots &&  \\  & 1 &&&  \\  1 &&&& 
\end{bmatrix},
$$                   
so $\f{so}(2n+1,\C)$ has Cartan subalgebra $\fh=\{{\rm diag}(t_1,t_2,\dots, t_n, 0, -t_n,\dots, -t_2, -t_1)\}$.  Then, by part (c) of Lemma \ref{lem:B1},  $\exp(\fnbar\cap\fnbar^{w_2^-})$ consists of matrices of the form 

$$
X=I_{2n+1} + \begin{bmatrix} 0 & 0  & 0 \\  x^t  & 0 & 0\\  0 & -x^\dagger & 0 \end{bmatrix}
+\frac12\begin{bmatrix}  0 & 0 & 0\\ 0 & 0 & 0 \\  -\IP{x}{x} & 0 & 0  \end{bmatrix},
$$
where $x=(x_1,x_2,\dots, x_{n+1},0,\dots,0)\in\C^{2n-1}$ and $x^\dagger=(0,\dots,0,x_{n+1},\dots,x_2,x_1)$ and $\IP{x}{x}=\sum x_jx_{2n-j}$.  For example, when $n=5$,
$$
\exp(\fnbar\cap\fnbar^{w_2^-})=\left\{\begin{bmatrix} 1& &&&&&&  \\ x_1 & 1 &&&&&&  \\
\vdots & & \ddots & & && \\  x_6 &&& 1 &&&&   \\  0 &&&& 1 &&& \\  0 &&&&& 1 && \\ 0 &&&&&& 1 & \\
 -\frac{\IP{x}{x}}{2} & 0 & 0 & 0 & -x_6 & \dots & -x_1 & 1 \end{bmatrix}\right\}.
$$
The slice is $S=\exp(\fnbar\cap\fnbar^{w_2^-})\cdot\fb$.  Viewing the Borel as the standard flag
$$\langle e_1 \rangle \subset  \langle e_1, e_2 \rangle \subset \cdots \subset \langle e_1, e_2,\dots,e_n \rangle,
$$
$S$ consists of  flags $F_1\subset F_2\subset \cdots \subset F_n$ with $F_j$ spanned by the first $j$ columns of matrices in 
$$
\exp(\fnbar\cap\fnbar^{w_2^-}) w_2^-\cdot I.
$$
Using the description of the Schubert varieties in type $B$ given in \cite{FultonPragacz98}, we may determine $S\cap Z_{w_1^+}$.  The realization we are using for $SO(2n+1,\C)$ gives a natural embedding of $W$ in $W(A_{2n})=S_{2n+1}$.  Then we may write
\begin{align*}
 w_1^+ &= (2n ,\,   2n-1 ,\,  \dots ,\,  n+3 ,\,  n+2 ,\,  2n+1 ,\, \,  n+1\, , \,\,\,  1 ,\, \hspace{14pt} n, \hspace{14pt}\,  n-1 ,\,  \dots ,\,  3 ,\,  2),  \\
   w_2^- &=  (2n ,\,  2n-1 ,\,  \dots ,\,  n+3 ,\hspace{12pt}  1 ,\hspace{16pt}  n+2 ,\,  \,  n+1, \, \,\,\,  n ,\,\,\,  2n+1 ,\,  n-1 ,\,  \dots ,\,  3 ,\,  2). 
\end{align*}
Following \cite{FultonPragacz98}, for any $w\in W$ (written in this form as a permutation in $S_{2n+1}$) we let $N_{p,q}(w)=\#\{i\leq p : w(i)\leq q\}$, then 
$$
Z_w=\{(F_j) : \dim(F_p\cap \langle e_1,\dots, e_q\rangle)\geq N_{p,q}(w), 1\leq p\leq 2n+1,1\leq q\leq n\}.
$$
Many of the inequalities are often redundant; in fact just one is needed for $w=w_1^+$.  We have
$$
Z_{w_1^+}=\{(F_j): \dim(F_{n+2}\cap\langle e_1\rangle)\geq N_{n+2,1}(w_1^+)=1 \}.
$$
In the $n=5$ example, elements of $S$ are
$$
\begin{bmatrix} &&& 1 &&& &&&& 0  \\
&&& x_1 & &&& &&&  1  \\    &&& x_ 2 &&&& && 1 &  \\   &&& x_ 3 &&&&& 1 &&  \\   
&&& x_4 &&& 1 & 0&&&  \\   &&& x_5 & & 1 &&&&&  \\   &&& x_6 & 1 &&&&&&  \\   
&& 1 & 0 &&&&&&  \\  & 1 && 0 &&&&&&&  \\   1 &&& 0 & &&&&&&  \\   
-x_1 & -x_2 & -x_3 &-\frac{\IP{x}{x}}{2} & -x_4 & -x_5 & -x_6 & 1 & 0 & 0 & 0 \\   
\end{bmatrix}.
$$
Such a matrix is in $Z_{w_1^+}$ if and only if the span of the  first $n+2=7$ columns contains $e_1$, i.e., $x_1=x_2=x_3=0$ and 
$$
\det\begin{bmatrix} x_4 & 0 & 0 & 1  \\   x_5 & 0 & 1 & 0  \\
                     x_6 & 1 & 0 & 0  \\   \frac{\IP{x}{x}}{2} & -x_4 & -x_5 & -x_6
    \end{bmatrix}=0.
$$
This hold if and only if $x_1=x_2=x_3=0$ and $\IP{x}{x}=0$.  For arbitrary $n$ the computation is virtually the same and we get
\begin{align*}
S\cap Z_{w_1^+}&=\{x\in \C^{n+1} : x_1=\cdots =x_{n-2}=0\text{ and }
  \IP{x}{x}=0\}  \\
  &=\{z\in \C^3 : z_1z_3+z_2^2=0\}
\end{align*}
This is the cone in $\C^3$.  It is known (e.g., the $\f{sl}(2,\C)$ case of \cite[Corollary 3.3]{EvensMirkovic99}) that the multiplicity of  the $IC$ sheaf of $ \cN$ at $\{0\}$ is $1$.  We conclude that $m_1=1$.

\begin{Thm}\label{thm:main-B}
For $G_\R=SO_e(2,2n-1)$ the characteristic cycles and leading term cycles of the highest weight Harish-Chandra modules are
\begin{align*}
  & CC(L_{w_k^+})=[\bar{T_{B_{w_{k}^+}}^*\fB}] + [\bar{T_{B_{w_{k+1}^-}}^*\fB}],
      1\leq k \leq n-1,  \\
  & LTC( L_{w_k^+} )=[\bar{T_{B_{w_k^+}}^*\fB}], 1 \leq k\leq  n-1.  \\
  & CC(L_{w_n^+})=LTC( L_{w_n^+} )=[\bar{T_{B_{w_n^+}}^*\fB}],   \\
  & CC(L_{w_k^-})= LTC( L_{w_k^-} )=[\bar{T_{B_{w_k^-}}^*\fB}], 1\leq k\leq n.  \\
\end{align*}
\end{Thm}

%%%%%%%%%%%%%%%%%%%%
\section{$Sp(2n,\R)$}\label{sec:sp} The characteristic cycles of highest weight Harish-Chandra modules were computed in \cite{BarchiniZierau17} in an inductive manner.  Here we state the result in a slightly more concise way.  

It is convenient to label the highest weight Harish-Chandra modules of infinitesimal character $\rho$ by clans.  In general, clans parametrize the $K$-orbits in $\fB$. A highest weight Harish-Chandra module is thus associated to the clan that specifies the closure of its support.  Of course, only certain $K$-orbit closures occur as the support of highest weight Harish-Chandra modules.  The set of  clans for such orbits will be denoted by $\Cl(n)$ (as in \cite{BarchiniZierau17}).  The elements of $\Cl(n)$ take the form 
$$
c=(c_1\,c_2\,c_3\, \dots\,c_n),\text{ with } c_i=+\text{ or }c_i\in\N.
$$
It is understood that the values of $c_i\in\N$ are immaterial, so two clans with natural numbers in the same places are considered the same.  Let $\fQ_c$ be the $K$-orbit corresponding to a clan $c\in\Cl(n)$.  We denote by $L_c$ the highest weight Harish-Chandra module of infinitesimal character $\rho$ having  $\bar{\fQ_c}$ as support.  It is  explained in \cite[\S2.2]{BarchiniZierau17} how to easily pass between $\Cl(n)$ and $\cW$, that is, how to determine the $w\in\cW$ so that $L_c$ has highest weight $-w\rho-\rho$.

In order to describe the characteristic cycles we first describe
\begin{equation*}
C_g^n(j):=\{c\in\Cl(n) : \mu(\bar{T_{\fQ_c}^*\fB})=\bar{\cO_j}\}, j=0,1,2,\dots,r;
\end{equation*}
we refer to this set as a `geometric cell'.  For each $c\in\Cl(n)$ define $(h_1,h_2,\dots,h_n)$ as follows:
\begin{equation*}
h_1=\begin{cases} 0, &c_n\in\N,  \\
                  1,  &c_n=+.
    \end{cases}
\end{equation*}
Then define the others inductively by
\begin{equation*}
  h_j=  \begin{cases}  h_{j-1},  &c_{n-j+1}\in \N,  \\
                       h_{j-1}+1,  &c_{n-j+1}=+\text{ and }h_{j-1}=j-1,  \\
                       h_{j-1}+2,  &c_{n-j+1}=+\text{ and }h_{j-1}\leq j-2.  \\
        \end{cases}
\end{equation*}
By Lemma 28 and equation (30) of \cite{BarchiniZierau17} $(c_{n-j+1}\,c_{n-j}\, \dots\, c_n)\in C_g^j(h_j)$.  (One should think of constructing $h$ from $c$ by working from right to left.)  For example, if $c=(1+2\;3+4++\;5)\in \Cl(9)$, then 
\begin{equation*}
\begin{array}{c|ccccccccc}
j & 9 & 8 & 7 & 6 & 5 & 4 & 3 & 2 &  1   \\
\hline
c &  1  & + & 2 & 3 & + & 4 & + & + &  5  \\
h_j& 7  & 7 & 5 & 5 & 5 & 3 & 3 & 2 &  0
\end{array}.
\end{equation*}      
(Note that $h_1=0$ since $c_9\in\N$, $h_2=2$ since $c_8=+$ and $h_1=0<j-2$,  $h_3=3$ since $c_7=+$ and $h_2=2=j-1$, etc.)  

The main result of \cite{BarchiniZierau17}, Theorem 70, gives $CC(L_c)$, for $c\in\Cl(n)$, as follows.  For $n=1$ the clans are $c=(+)$ and $c=(1)$, and $CC(L_c)$ is just $\bar{T_{\fQ_c}^*\fB}$ in each case.  For $n>1$, $c=(+\,c')$ or $c=(1\, c')$, for some $c'\in\Cl(n-1)$.  Then $CC(L_c)$ is described in terms of $c'$ as follows.

Let us define $D(c)$ by 
$$
CC(L_c)=\sum_{d\in D(c)}\bar{T_{\fQ_d}^*\fB}.
$$
Suppose that $D(c')$ is known, then
for $c=(+\,c')$,
\begin{subequations}\label{eqn:main}
\begin{align}
D(c)&=\{(+\,c_1'): c_1' \in D(c')\}, 
\intertext{and for $c=(1\,c')$ there are two cases:}
D(c)&=\{(1\,c_1'): c_1' \in D(c')\}, \text{ if $n$ is odd, or $n$ is even and }c'\notin C_g^{n-1}(n-1)\text{  and } \\
D(c)&=\{(1\,c_1'): c_1' \in D(c')\}\cup \{(+\,c_1'): c_1' \in D(c')\}, 
  \text{ if } n\text{ is even and }c'\in C_g^{n-1}(n-1).  
\end{align}
\end{subequations}

It follows that the characteristic cycle of $L_c$ is irreducible exactly when there is no $j=1,2,\dots,n-1$ so that 
\begin{subequations}\label{eqn:J}
\begin{equation}
   h_{j+1}=h_j=j, \text{ with }j\text{ odd }
\end{equation}
or equivalently, there is no $j=1,2,\dots,n-1$ so that 
\begin{equation}
   h_j=j\text{ and } c_{n-j}=\in \N \text{ with }j\text{ odd}
\end{equation}
\end{subequations}
(and necessarily $c_j=+$).
If we let $J(c)$ be the set of $j=1,2,\dots,n-1$ satisfying (\ref{eqn:J}), then the number of terms in $CC(L_c)$ is $2^{\#J(c)}$, in fact
$$
  D(c)=\{d\in \Cl(n): d_i=c_i\text{ if } i\notin J(c)\}.
$$
In the above example, $h_4=h_3=3$ and $h_6=h_5=5$, so
\begin{equation*}
 D(c):= \{(1+2\; 3+4++\,5),  (1+2\; 3+4\;5+ 6),  (1+2\; 3\;4\;5++\,6),  (1+2\; 3\;4\;5\;6+\,7)   \}
\end{equation*}
\begin{Example}  The most complicated characteristic cycles are for 
 \begin{equation*}
   c=\begin{cases}  (1+2+3\,\dots\,+l+),  &n=2l,  \\
                  (+1+2+3\,\dots\,+l+),  &n=2l+1.
     \end{cases}
   \end{equation*}
   There are $2^l$ terms in  $CC(L_c)$ in each case.
\end{Example}

The description of characteristic cycles given above allows us to prove the following proposition; the proof is given in Appendix \ref{appendix:B}.
\begin{Prop}\label{prop:count}  The number of irreducible highest weight Harish-Chandra modules of infinitesimal character $\rho$ having irreducible characteristic cycle is given by
 \begin{equation*}
  N(2l)={2l+1\choose l} \text{ and } N(2l+1)=2N(2l).
 \end{equation*}
When $n=2l$, the proportion of highest weight Harish-Chandra modules with irreducible characteristic cycle is 
\begin{equation*}
 \frac{N(2l)}{2^{2l}}=\frac{1}{2^{2l}} {2l+1\choose l}\sim \frac{2}{\sqrt{\pi}}\frac{1}{\sqrt{l}}.
\end{equation*}
\end{Prop}

\vspace{1cm}

%\newpage

%%%%%%%%%%%%%%%%%%%%
\appendix

%%%%%%%%%%%%%%%%%%%%
\section{Data for the exceptional cases}\label{appendix:A}
Here we give the  relevant data for the computations of Section \ref{sec:exceptional}.  The elements of $\cW$  are given in Table 3 by specifying $w_1,\dots,w_{27}\in\cW$ for $E_6$ by giving each $w_i$ as the long element of $W_c$ followed by the product of the simple reflections listed.

\begin{table}[H]
\scalebox{.8}{\mbox{\hspace{-1.5cm}
$
\setlength{\extrarowheight}{3pt}
\begin{array}{ll | ll | ll}
i &  \text{Red. expr. for }w_i  \quad {}   &
i &  \text{Red. expr. for }w_i  \quad       & 
i & \text{Red. expr. for }w_i  \quad      \\ 
\hline  
1   &        & 
10  &     [ 1, 3, 4, 5, 6, 2]    & 
19  &     [ 1, 3, 4, 5, 6, 2, 4, 5, 3, 4]   \\
2   &     [1]     & 
11  &     [1, 3, 4, 5, 2, 4, 3   & 
20  &     [1, 3, 4, 5, 6, 2, 4, 5, 3, 4, 1]   \\
3   &     [1,3]    & 
12  &     [1, 3, 4, 5, 6, 2, 4]    & 
21  &     [ 1, 3, 4, 5, 6, 2, 4, 5, 3, 4, 2]      \\
4   &     [1,3,4]    & 
13  &     [1, 3, 4, 5, 2, 4, 3, 1]  & 
22  &     [1, 3, 4, 5, 6, 2, 4, 5, 3, 4, 2, 1]   \\
5   &     [1,3,4,2]    & 
14  &     [1, 3, 4, 5, 6, 2, 4, 3]    & 
23  &     [1, 3, 4, 5, 6, 2, 4, 5, 3, 4, 1, 3]    \\
6   &     [1,3,4,5]     & 
15  &     [1, 3, 4, 5, 6, 2, 4, 5]     & 
 24  &     [ 1, 3, 4, 5, 6, 2, 4, 5, 3, 4, 2, 1, 3]   \\
7  &     [1,3,4,5,2]    & 
16  &     [1, 3, 4, 5, 6, 2, 4, 3, 1]   & 
25  &     [1, 3, 4, 5, 6, 2, 4, 5, 3, 4, 2, 1, 3, 4]    \\
 8  &     [1,3,4,5,6]     & 
 17  &     [1, 3, 4, 5, 6, 2, 4, 5, 3]    & 
26  &     [1, 3, 4, 5, 6, 2, 4, 5, 3, 4, 2, 1, 3, 4, 5]  \\
9  &     [1, 3, 4, 5, 2, 4]  & 
18  &     [1, 3, 4, 5, 6, 2, 4, 5, 3, 1]   & 
27  &     [1, 3, 4, 5, 6, 2, 4, 5, 3, 4, 2, 1, 3, 4, 5, 6]   
\end{array}
$
}}
\medskip
\caption{$\cW$ for $E_6$}
\label{tab:w-e6}
\end{table}

\newpage
The next table specifies $w_1,\dots,w_{56}\in\cW$ for $E_7$, again by giving each $w_i$ as the long element of $W_c$ followed by the product of the simple reflections listed.

\begin{table}[H]\label{table:e7-w}
%\begin{sideways}
%\begin{table}
%\begin{minipage}{\textheight}
%\renewcommand{\arraystretch}{1.75}
\scalebox{.8}{\mbox{\hspace{0cm}
$
\setlength{\extrarowheight}{3pt}
\begin{array}{ll|ll}
i       &  \text{Reduced expression for }w_i  \qquad\qquad   &  i       &  \text{Reduced expression for }w_i    \\
\hline  
 1    &         &  29    &     7, 6, 5, 4, 2, 3, 1, 4, 3, 5, 4, 2, 6, 5      \\ 
 2    &     7   &  30    &     7, 6, 5, 4, 2, 3, 1, 4, 3, 5, 4, 2, 6, 7  \\
 3    &     7, 6  &  31    &     7, 6, 5, 4, 2, 3, 1, 4, 3, 5, 4, 6, 5, 7   \\ 
 4    &     7, 6, 5  &  32    &     7, 6, 5, 4, 2, 3, 1, 4, 3, 5, 4, 2, 6, 5, 4  \\ 
 5    &     7, 6, 5, 4  &  33    &     7, 6, 5, 4, 2, 3, 1, 4, 3, 5, 4, 2, 6, 5, 7   \\ 
 6    &     7, 6, 5, 4, 2 &  34    &     7, 6, 5, 4, 2, 3, 1, 4, 3, 5, 4, 6, 5, 7, 6    \\ 
 7    &     7, 6, 5, 4, 3  &  35    &     7, 6, 5, 4, 2, 3, 1, 4, 3, 5, 4, 2, 6, 5, 4, 3    \\ 
 8    &     7, 6, 5, 4, 2, 3  &  36    &     7, 6, 5, 4, 2, 3, 1, 4, 3, 5, 4, 2, 6, 5, 4, 7    \\ 
 9    &     7, 6, 5, 4, 3, 1  &  37    &     7, 6, 5, 4, 2, 3, 1, 4, 3, 5, 4, 2, 6, 5, 7, 6   \\ 
10    &     7, 6, 5, 4, 2, 3, 1 &  38    &     7, 6, 5, 4, 2, 3, 1, 4, 3, 5, 4, 2, 6, 5, 4, 3, 1   \\ 
11    &     7, 6, 5, 4, 2, 3, 4 &  39    &     7, 6, 5, 4, 2, 3, 1, 4, 3, 5, 4, 2, 6, 5, 4, 3, 7    \\ 
12    &     7, 6, 5, 4, 2, 3, 1, 4 &  40    &     7, 6, 5, 4, 2, 3, 1, 4, 3, 5, 4, 2, 6, 5, 4, 7, 6  \\ 
13    &     7, 6, 5, 4, 2, 3, 4, 5 &  41    &     7, 6, 5, 4, 2, 3, 1, 4, 3, 5, 4, 2, 6, 5, 4, 3, 1, 7   \\ 
14    &     7, 6, 5, 4, 2, 3, 1, 4, 3 &  42    &     7, 6, 5, 4, 2, 3, 1, 4, 3, 5, 4, 2, 6, 5, 4, 3, 7, 6    \\ 
15    &     7, 6, 5, 4, 2, 3, 1, 4, 5  &  43    &     7, 6, 5, 4, 2, 3, 1, 4, 3, 5, 4, 2, 6, 5, 4, 7, 6, 5   \\ 
16    &     7, 6, 5, 4, 2, 3, 4, 5, 6  &  44    &     7, 6, 5, 4, 2, 3, 1, 4, 3, 5, 4, 2, 6, 5, 4, 3, 1, 7, 6    \\ 
17    &     7, 6, 5, 4, 2, 3, 1, 4, 3, 5   &  45    &     7, 6, 5, 4, 2, 3, 1, 4, 3, 5, 4, 2, 6, 5, 4, 3, 7, 6, 5 \\ 
18    &     7, 6, 5, 4, 2, 3, 1, 4, 5, 6  &  46    &     7, 6, 5, 4, 2, 3, 1, 4, 3, 5, 4, 2, 6, 5, 4, 3, 1, 7, 6, 5     \\ 
19    &     7, 6, 5, 4, 2, 3, 4, 5, 6, 7  &  47    &     7, 6, 5, 4, 2, 3, 1, 4, 3, 5, 4, 2, 6, 5, 4, 3, 7, 6, 5, 4  \\ 
20    &     7, 6, 5, 4, 2, 3, 1, 4, 3, 5, 4  &  48    &     7, 6, 5, 4, 2, 3, 1, 4, 3, 5, 4, 2, 6, 5, 4, 3, 1, 7, 6, 5, 4    \\ 
21    &     7, 6, 5, 4, 2, 3, 1, 4, 3, 5, 6  &  49    &     7, 6, 5, 4, 2, 3, 1, 4, 3, 5, 4, 2, 6, 5, 4, 3, 7, 6, 5, 4, 2    \\ 
22    &     7, 6, 5, 4, 2, 3, 1, 4, 5, 6, 7  &  50    &     7, 6, 5, 4, 2, 3, 1, 4, 3, 5, 4, 2, 6, 5, 4, 3, 1, 7, 6, 5, 4, 2   \\ 
23    &     7, 6, 5, 4, 2, 3, 1, 4, 3, 5, 4, 2  &  51    &     7, 6, 5, 4, 2, 3, 1, 4, 3, 5, 4, 2, 6, 5, 4, 3, 1, 7, 6, 5, 4, 3     \\ 
24    &     7, 6, 5, 4, 2, 3, 1, 4, 3, 5, 4, 6  &  52    &     7, 6, 5, 4, 2, 3, 1, 4, 3, 5, 4, 2, 6, 5, 4, 3, 1, 7, 6, 5, 4, 2, 3    \\ 
25    &     7, 6, 5, 4, 2, 3, 1, 4, 3, 5, 6, 7  &  53    &     7, 6, 5, 4, 2, 3, 1, 4, 3, 5, 4, 2, 6, 5, 4, 3, 1, 7, 6, 5, 4, 2, 3, 4   \\ 
26    &     7, 6, 5, 4, 2, 3, 1, 4, 3, 5, 4, 2, 6  & 54    &     7, 6, 5, 4, 2, 3, 1, 4, 3, 5, 4, 2, 6, 5, 4, 3, 1, 7, 6, 5, 4, 2, 3, 4, 5     \\ 
27    &     7, 6, 5, 4, 2, 3, 1, 4, 3, 5, 4, 6, 5  & 55    &     7, 6, 5, 4, 2, 3, 1, 4, 3, 5, 4, 2, 6, 5, 4, 3, 1, 7, 6, 5, 4, 2, 3, 4, 5, 6   \\
28    &     7, 6, 5, 4, 2, 3, 1, 4, 3, 5, 4, 6, 7  & 56    &     7, 6, 5, 4, 2, 3, 1, 4, 3, 5, 4, 2, 6, 5, 4, 3, 1, 7, 6, 5, 4, 2, 3, 4, 5, 6, 7    
\end{array}
$
}}
\smallskip
\caption{$\cW$ for $E_7$}
\end{table}

\newpage

The remaining two tables list the relevant data for the $E_7$ case.\\
\begin{table}[H]\label{table}\label{table:e7-data}
\scalebox{.84}{\mbox{\hspace{0cm}
$
\setlength{\extrarowheight}{3pt}
\begin{array}{llccc|clccc}
i& \quad \tau(w_i)         &  \dim(Z_w)  & AV & \text{possible CC} & 
\quad i\qquad& \quad \tau(w_i)         &  \dim(Z_w)  & AV & \text{possible CC}  \\
\hline  
1   &    \{1, 2, 3, 4, 5, 6\}  &      36  &            \cO_3  &                          \{1\}    &
29  &       \{1, 2, 3, 5, 6\}  &      50  &              \cO_2  &                     \{25, 29\}    \\ 
 2  &    \{1, 2, 3, 4, 5, 7\}  &      37  &            \cO_3  &                          \{2\}    &
30  &    \{1, 2, 3, 4, 6, 7\}  &      50  &              \cO_2  &                         \{30\}    \\ 
 3  &    \{1, 2, 3, 4, 6, 7\}  &      38  &            \cO_3  &                          \{3\}    &
31  &       \{1, 3, 4, 5, 7\}  &      50  &              \cO_2  &                     \{23, 31\}    \\ 
 4  &    \{1, 2, 3, 5, 6, 7\}  &      39  &            \cO_3  &                          \{4\}    & 
32  &       \{1, 2, 4, 5, 6\}  &      51  &              \cO_2  &                     \{22, 32\}    \\ 
 5  &       \{1, 4, 5, 6, 7\}  &      40  &            \cO_3  &                          \{5\}    & 
33  &       \{1, 2, 3, 5, 7\}  &      51  &              \cO_2  &                 \{23, 25, 33\}    \\ 
 6  &    \{1, 2, 4, 5, 6, 7\}  &      41  &            \cO_3  &                          \{6\}    & 
34  &    \{1, 3, 4, 5, 6, 7\}  &      51  &              \cO_2  &                         \{34\}    \\ 
 7  &       \{3, 4, 5, 6, 7\}  &      41  &            \cO_3  &                          \{7\}    & 
35  &       \{2, 3, 4, 5, 6\}  &      52  &              \cO_2  &                     \{19, 35\}    \\ 
 8  &       \{2, 3, 5, 6, 7\}  &      42  &            \cO_3  &                       \{4, 8\}    & 
36  &       \{1, 2, 4, 5, 7\}  &      52  &              \cO_2  &                 \{22, 23, 36\}    \\ 
 9  &    \{1, 3, 4, 5, 6, 7\}  &      42  &            \cO_3  &                          \{9\}    & 
37  &    \{1, 2, 3, 5, 6, 7\}  &      52  &              \cO_2  &                     \{25, 37\}    \\ 
10  &    \{1, 2, 3, 5, 6, 7\}  &      43  &            \cO_3  &                      \{4, 10\}    & 
38  &    \{1, 2, 3, 4, 5, 6\}  &      53  &              \cO_2  &                         \{38\}    \\ 
11  &       \{2, 3, 4, 6, 7\}  &      43  &            \cO_3  &                      \{3, 11\}    & 
39  &       \{2, 3, 4, 5, 7\}  &      53  &              \cO_2  &                 \{19, 23, 39\}    \\ 
12  &       \{1, 2, 4, 6, 7\}  &      44  &            \cO_3  &                   \{3, 6, 12\}    & 
40  &       \{1, 2, 4, 6, 7\}  &      53  &              \cO_2  &                 \{22, 30, 40\}    \\ 
13  &       \{2, 3, 4, 5, 7\}  &      44  &            \cO_3  &                      \{2, 13\}    & 
41  &    \{1, 2, 3, 4, 5, 7\}  &      54  &              \cO_2  &                     \{23, 41\}    \\ 
14  &    \{1, 2, 3, 4, 6, 7\}  &      45  &            \cO_3  &                      \{3, 14\}    & 
42  &       \{2, 3, 4, 6, 7\}  &      54  &              \cO_2  &                 \{19, 30, 42\}    \\ 
15  &       \{1, 2, 4, 5, 7\}  &      45  &            \cO_3  &                   \{2, 6, 15\}    &
43  &    \{1, 2, 4, 5, 6, 7\}  &      54  &              \cO_2  &                     \{22, 43\}    \\ 
16  &       \{2, 3, 4, 5, 6\}  &      45  &            \cO_3  &                      \{1, 16\}    & 
44  &    \{1, 2, 3, 4, 6, 7\}  &      55  &              \cO_2  &                     \{30, 44\}    \\ 
17  &       \{1, 2, 3, 5, 7\}  &      46  &            \cO_3  &               \{2, 4, 10, 17\}    &
45  &       \{2, 3, 5, 6, 7\}  &      55  &              \cO_2  &             \{19, 25, 37, 45\}    \\ 
18  &       \{1, 2, 4, 5, 6\}  &      46  &            \cO_3  &                   \{1, 6, 18\}    & 
46  &    \{1, 2, 3, 5, 6, 7\}  &      56  &              \cO_2  &                 \{25, 37, 46\}    \\ 
19  &    \{2, 3, 4, 5, 6, 7\}  &      46  &              \cO_2  &                         \{19\}   & 
47  &       \{3, 4, 5, 6, 7\}  &      56  &              \cO_2  &                 \{19, 34, 47\}    \\ 
20  &       \{1, 3, 4, 5, 7\}  &      47  &            \cO_3  &                   \{2, 9, 20\}    & 
48  &       \{1, 4, 5, 6, 7\}  &      57  &              \cO_2  &             \{22, 34, 43, 48\}    \\ 
21  &       \{1, 2, 3, 5, 6\}  &      47  &            \cO_3  &               \{1, 4, 10, 21\}    & 
49  &    \{2, 3, 4, 5, 6, 7\}  &      57  &              \cO_1  &                         \{49\}    \\ 
22  &    \{1, 2, 4, 5, 6, 7\}  &      47  &              \cO_2  &                         \{22\}   &
50  &    \{1, 2, 4, 5, 6, 7\}  &      58  &              \cO_1  &                         \{50\}    \\ 
23  &    \{1, 2, 3, 4, 5, 7\}  &      48  &              \cO_2  &                         \{23\}   & 
51  &    \{1, 3, 4, 5, 6, 7\}  &      58  &              \cO_1  &                         \{51\}    \\ 
24  &          \{1, 3, 4, 6\}  &      48  &            \cO_3  &           
 \{1, 3, 9, 14, 24\}    & 
 52  &    \{1, 2, 3, 5, 6, 7\}  &      59  &              \cO_1  &                         \{52\}    \\ 
25  &    \{1, 2, 3, 5, 6, 7\}  &      48  &              \cO_2  &                         \{25\}    & 
53  &    \{1, 2, 3, 4, 6, 7\}  &      60  &              \cO_1  &                         \{53\}    \\ 
26  &       \{1, 2, 3, 4, 6\}  &      49  &              \cO_2  &                         \{26\}    & 
54  &    \{1, 2, 3, 4, 5, 7\}  &      61  &              \cO_1  &                         \{54\}    \\ 
27  &       \{1, 3, 4, 5, 6\}  &      49  &              \cO_2  &                         \{27\}   &
55  &    \{1, 2, 3, 4, 5, 6\}  &      62  &              \cO_1  &                         \{55\}    \\ 
28  &       \{1, 3, 4, 6, 7\}  &      49  &              \cO_2  &                         \{28\}    & 
56  &  \{1, 2, 3, 4, 5, 6, 7\}  &      63  &             \cO_0  &                         \{56\}    
  \end{array}
$}}
\caption{$E_7$ Data}
\end{table}

\begin{table}[H]\label{table:e7-tab}
\scalebox{.84}{\mbox{\hspace{0cm}
$
\setlength{\extrarowheight}{3pt}
\begin{array}{cc|cc|cc|cc|cc|cc}
w & \T_{13}(w)       &  w & \T_{34}(w)       &w & \T_{24}(w)       &w & \T_{45}(w)       &w & \T_{56}(w)   &w & \T_{67}(w)            \\
\hline  
               7  &   4  &   5  &   4  &   5  &   4  &   4  &   3  &   3  &   2  &   2  &   1  \\
            8  &   8  &   6  &   8  &   7  &   8  &   8  &   11  &   11  &   13  &   13  &   16  \\
         11  &   10  &   12  &   10  &   9  &   10  &   10  &   12  &   12  &   15  &   15  &   18  \\
         13  &   17  &   15  &   17  &   20  &   17  &   17  &   14  &   14  &   17  &   17  &   21  \\
         16  &   21  &   18  &   21  &   24  &   21  &   21  &   24  &   24  &   20  &   20  &   24  \\
         19  &   25  &   22  &   25  &   27  &   29  &   25  &   28  &   26  &   23  &   23  &   26  \\
         35  &   29  &   32  &   29  &   28  &   25  &   29  &   26  &   28  &   31  &   31  &   27 \\
         39  &   33  &   36  &   33  &   31  &   33  &   33  &   30  &   30  &   33  &   33  &   29  \\
         42  &   37  &   40  &   37  &   34  &   37  &   37  &   40  &   40  &   36  &   36  &   32  \\
         45  &   45  &   43  &   45  &   47  &   45  &   45  &   42  &   42  &   39  &   39  &   35  \\
         47  &   46  &   48  &   46  &   48  &   46  &   46  &   44  &   44  &   41  &   41  &   38  \\
         49  &   52  &   50  &   52  &   51  &   52  &   52  &   53  &   53  &   54  &   54  &   55  

\end{array}
$

}  }
\caption{$\T_{\ga,\gb}$'s for $E_7$}
\end{table}
\medskip

\section{}\label{appendix:B}

The proof of Proposition \ref{prop:count} is now given.  Recall that $G_\R=Sp(2n,\R)$ and the characteristic cycles are described in \S\ref{sec:sp}.  The proposition counts  the number of highest weight Harish-Chandra modules of infinitesimal character $\rho$ having irreducible characteristic cycle.  Let $N(n)$ be the number of such representations for $Sp(2n,\R)$.

We consider an array as follows.

\scalebox{.9}{\mbox{\hspace{1.5cm}
$
\begin{array}{ccllccccccc}
 \text{col}\#  & n &\quad  \dots && 5
 & 4 & 3 & 2 & 1 & 0     \\
\hline
 & n & \quad\dots\quad && 5 & 4 & 3 & 2 & 1 & 0  \\  
 & n-1 & \quad \dots && 4 & 3 & 2 & 1 & 0  &  \\  
 & n-2 & \quad \dots  && 3 & 2 & 1 & 0  &  &  \\  
 & n-3 & \quad \dots  && 2  & 1 & 0  &     &  & \\  
 & n-4 & \quad \dots  &&  1  & 0 &  &  &  &  \\  
 & n-5 & \quad \dots  &&   0  &  &  &  &  &  \\  
 & \vdots &  &  \iddots &&  &  &  &   & \\  

 & 1 & 0 &   &&  &  &  &   &  \\  
 & 0\quad &   &   &&  &  &  &   &   
 \end{array}
$
}}

\noindent Note that the columns are numbered from right to left.   As clans for $Sp(2n,\R)$ are built up from clans for $Sp(2(n-1),\R)$ by $c'\leadsto(+\,c')$ or $(1\,c')$, the number of elements in $C_g^m(j)$ is the number of paths connecting entries of the array so that the path
\begin{enumerate}[label=\upshape{(\roman*)}]
 \item starts at the $0$ in column $\#0$,
 \item ends at the entry labelled with $j$ in column $\#m$,
 \item goes west or southwest from any top row entry and goes northwest or southwest from any other entry.
\end{enumerate}
The southwest moves represent $c'\leadsto (1\,c')$ and the west/northwest moves represent $c'\leadsto(+\,c')$.  This fact is a restatement of \cite[Lemma 28 and equation (30)]{BarchiniZierau17}.

By (\ref{eqn:main}), $CC(L_c)$ is irreducible if and only if the path corresponding to $c$ does not pass through the entry labelled with $2k-1$ in column $\#2k$, for any $k$.  But such a path cannot pass through \emph{any} odd label, other than those in the first row.  We therefore see that when $n$ is even,  the number of $c\in C_g^n(j)$ with $CC(L_c)$ irreducible is the number of paths (as described above) through 

\scalebox{.9}{\mbox{\hspace{1.5cm}
$\begin{array}{crcllcccccccccc}
 \text{col}\#  & n &&\quad  \dots & 5 &  & 4 &  & 3 &  & 2 &  & 1 &     &  0   \\
\hline
 & n & \leftarrow  && 5 & \leftarrow & 4  & \leftarrow & 3 & \leftarrow & 2 & \leftarrow & 1 & \leftarrow & 0  \\  
& & \nwarrow &   &  & \swarrow &  &  \nwarrow &  & \swarrow &  & \nwarrow  &  & \swarrow &     \\  
& &  &   & 4 &  &  &   & 2 &  &  &  & 0 &  &     \\  
&  &\swarrow &         &  & \nwarrow &  & \swarrow  &  & \nwarrow &  & \swarrow &  &  &     \\  
&   n-2  & &   &  &  & 2 &   &  &  & 0 &  &  &  &     \\  
& &  \nwarrow        &   &  &  \swarrow &  &   \nwarrow &  & \swarrow &  &  &  &  &     \\  
& &          &   &  &  &                 &  &  &  &  &  &  &  &     \\  
& &   \vdots       &   &  &   & & \iddots  &  &  &  &  &  &  &     \\  
& &          &   &  &  &  &   &  &  &  &  &  &  &     \\  
& & \swarrow &   &  &   &  &   &  &  &  &  &  &  &     \\  
& 2 &  &   &   &         &  &   &  &  &  &  &  &  &     \\  
& & \nwarrow &   \iddots &  &  &  &   &  &  &  &  &  &  &     \\  
& &  &    &  &  &  &   &  &  &  &  &  &  &     \\  
& & \swarrow &    &  &  &  &   &  &  &  &  &  &  &     \\  
& 0 &  &    &  &  &  &   &  &  &  &  &  &  &     
 \end{array}
$
}}

\noindent Rearranging slightly (by moving the odd entries up  and inserting  a node on the far right) this is the number of paths through 

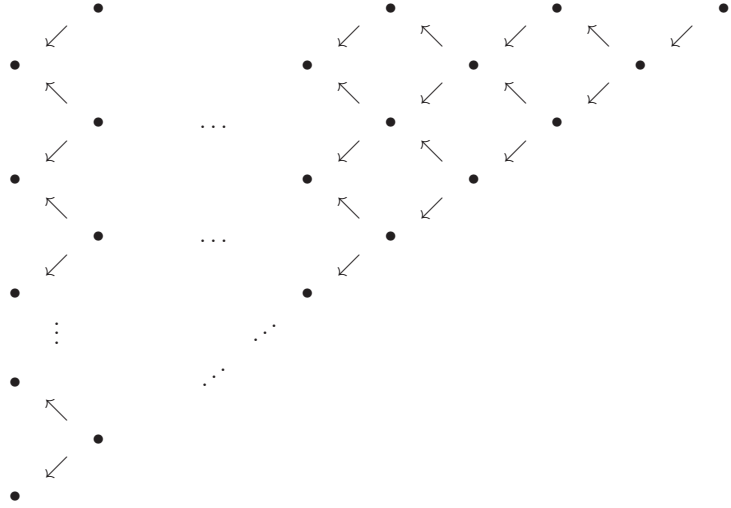
\scalebox{.9}{\mbox{\hspace{1.5cm}
$\begin{array}{ccc ccc ccc ccc ccc ccc ccc}
 &  & \bullet & &&&&&  &    & \bullet &  &  &  & \bullet  &  &  &  &  \bullet \\  
 & \swarrow &      &&& &    && & \swarrow &   & \nwarrow &  & \swarrow &  & \nwarrow &  & \swarrow &   \\  
 \bullet & &&  &  &&&&   \bullet &  &  &  & \bullet  &  &  &  &  \bullet & &   \\  
 & \nwarrow &       &   & &&&&&  \nwarrow &   & \swarrow &  & \nwarrow &  & \swarrow &  &  &   \\  
 &  & \bullet & &&&\dots&&  &    & \bullet &  &  &  & \bullet  &  &  &  &   \\  
 & \swarrow &      &&& &   & &&  \swarrow &   & \nwarrow &  & \swarrow &  & &  & &   \\  
 \bullet & &&  &  &&&&   \bullet &  &  &  & \bullet  &  &  &  &  & &   \\  
 & \nwarrow &       &   & &&&&&  \nwarrow &   & \swarrow &  &  &  &  &  &  &   \\  
 &  & \bullet & &&&\dots&&  &    & \bullet &  &  &  &   &  &  &  &   \\  
 & \swarrow &       &&&&   & &&  \swarrow &   &  &  &  &  &  &  &  &   \\  
 \bullet &    &&&&&   && \bullet &  &  &  &   &  &  &  &  & &   \\  
 & \vdots     &&&&&   & \iddots& & &  &       &   & && &  \\
 \bullet &    &&&&&   \iddots&  & & &  &       &   & && &  \\
  & \nwarrow &       &&&   & && & &  &       &   & && &  \\
 &  & \bullet & &&  &&&  &  & &  &  &  &   &  &  &  &   \\  
 & \swarrow &       &&&   & && & &   &  &  &  &  & &  &  &   \\  
 \bullet & &&  &  &&&  &  &  &  &  &   &  &  &  &   & &   \\  
 \end{array}
$
}}

\noindent with $l+1$ nodes across the top.

Let $N(n,j)$ be the number of $c\in C_g^n(j)$ so that $CC(L_c)$ is irreducible.  Then $N(n)=\sum_{j=1}^n N(n,j)$.  When $n$ is even
$N(n,j)$ is the number of paths in the above diagram from the upper rightmost node to  node $j$ in column $\#n$.

\begin{Lem}
For $n=2l$, $\displaystyle N(n,n)=\frac{1}{l+2}{2l+2\choose l+1}$, the Catalan number $c(l+1).$
\end{Lem}

\begin{proof}
 To count the paths from the rightmost node to the top node in the leftmost column, we may omit part of the array and consider only

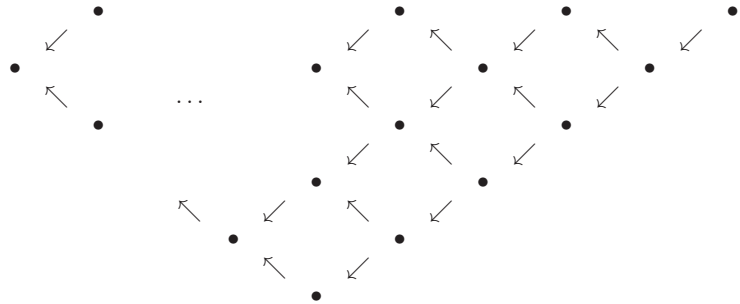
\scalebox{.9}{\mbox{\hspace{1.5cm}
$\begin{array}{ccc ccc ccc ccc ccc ccc ccc}
 &  & \bullet & &&&&&  &    & \bullet &  &  &  & \bullet  &  &  &  &  \bullet \\  
 & \swarrow &      &&& &    && & \swarrow &   & \nwarrow &  & \swarrow &  & \nwarrow &  & \swarrow &   \\  
 \bullet & &&  &  &&&&   \bullet &  &  &  & \bullet  &  &  &  &  \bullet & &   \\  
 & \nwarrow &       &   & &  \dots &&&&  \nwarrow &   & \swarrow &  & \nwarrow &  & \swarrow &  &  &   \\  
 &  & \bullet & &&&&&  &    & \bullet &  &  &  & \bullet  &  &  &  &   \\  
 &  &      &&& &   & &&  \swarrow &   & \nwarrow &  & \swarrow &  & &  & &   \\  
 & &&  &  &&&&   \bullet &  &  &  & \bullet  &  &  &  &  & &   \\  
 &  &       &   &&\nwarrow  && \swarrow     &&  \nwarrow &   & \swarrow &  &  &  &  &  &  &   \\  
 &  &  &&&&  \bullet &&  &    & \bullet &  &  &  &   &  &  &  &   \\  
 &  &       &&&&   &  \nwarrow &&  \swarrow &   &  &  &  &  &  &  &  &   \\  
 &    &&&&&   && \bullet &  &  &  &   &  &  &  &  & &   \\  
 \end{array}
$
}}  \\
with $l+1$ nodes along the top.  The number of such paths is $\frac{1}{l+2}{2l+2\choose l+1}$ (see e.g., \cite[Ch.~6, page 221]{Stanley99}).
\end{proof}
For example, when $n=4$ ($l=2$) the array is 

 \scalebox{.9}{\mbox{\hspace{1.5cm}
 $\begin{array}{c ccc ccc ccc ccc}
   &    & \bullet &  &  &  & \bullet  &  &  &  &  \bullet \\  
  & \swarrow &   & \nwarrow &  & \swarrow &  & \nwarrow &  & \swarrow &   \\  
    \bullet &  &  &  & \bullet  &  &  &  &  \bullet & &   \\  
 &  \nwarrow &   & \swarrow &  & \nwarrow &  & \swarrow &  &  &   \\  
   &    & \bullet &  &  &  & \bullet  &  &  &  &   \\  
  &   &   & \nwarrow &  & \swarrow &  & &  & &   \\  
  &  &  &  & \bullet  &  &  &  &  & &   \\  
 \end{array}.
$
}} \\
There are five paths.

It is easy to check that 
\begin{equation*}
 N(n,j)=N(n-1,j) + N(n-1,j-2), 0\leq j \leq n-1,
\end{equation*}
for any $n$.  This is necessarily $0$ if $j$ is odd.  Also
\begin{equation*}
 N(n,n)=\begin{cases}
         N(n-1,n-1)+N(n-1,n-2), & n\text{ even},  \\
         N(n-1,n-1),  & n\text{ odd}.
        \end{cases}
\end{equation*}
Using this we easily have the following.
\begin{Lem}
 For any $n$
 \begin{equation*}
 N(n)=\begin{cases}
         2N(n-1), & n\text{ odd},  \\
        4N(n-2)-N(n-2,n-2),  &n \text{ even}.
        \end{cases}
\end{equation*}
\end{Lem}

Now we prove the proposition.  The formula $N(2l)={2l+1 \choose  l}$ is proved by induction.  The $l=1$ case is immediate by counting paths (or by referring to \cite{BarchiniZierau17}).  Assume the formula for $l-1$ in place of $l$.
\begin{align*}
 N(2l)&=  4N(2l-2)-N(2l-2,2l-2)  \\
    &=4{2l-1\choose l-1}  - c(l)  \\
    &= \frac{4(2l-1)!}{(l-1)!\,l!}  -  \frac{1}{l+1}\frac{(2l)!}{l!\,l!}  \\
    &={2l+1 \choose l}.
\end{align*}

Since the number of clans in $\Cl(n)$ is $2^n$, the fraction of those for $n=2l$ having irreducible characteristic cycle is 
$$\frac{1}{2^{2l}}{2l+1\choose l}.
$$
Since $m!\sim\sqrt{2\pi m}\left(\frac{m}{e}\right)^m$ (Stirling's formula), $(m+1)^{m+1}\sim e\, m^{m+1}$ and $\sqrt{2m+1}\sim\sqrt{2}\sqrt{m+1}$ we easily obtain 
\begin{equation*}
 \frac{1}{2^{2l}}{2l+1 \choose l}  \sim \frac{2}{\sqrt{\pi}}\frac{1}{\sqrt{l}}.
\end{equation*}

\vspace{-14pt}\hfill{$\square$}

\providecommand{\bysame}{\leavevmode\hbox to3em{\hrulefill}\thinspace}
\providecommand{\MR}{\relax\ifhmode\unskip\space\fi MR }
\providecommand{\MRhref}[2]{%
  \href{http://www.ams.org/mathscinet-getitem?mr=#1}{#2}
}
\providecommand{\href}[2]{#2}

%\bibliographystyle{amsplain}
%\bibliography{../references}

\begin{thebibliography}{10}

\bibitem{BarbaschVogan82}
D.~Barbasch and D.~Vogan, \emph{Primitive ideals and orbital integrals in
  complex classical groups}, Math.~Ann. \textbf{259} (1982), 153--199.

\bibitem{BarchiniZierau17}
L.~Barchini and R.~Zierau, \emph{Characteristic cycles of highest weight
  {H}arish-{C}handra modules for {${\rm Sp}(2n,{\bf R})$}}, Transform. Groups
  \textbf{22} (2017), no.~3, 591--630.

\bibitem{BeilinsonBernstein81}
A.\ Beilinson and J.~Bernstein, \emph{Localisation de $\fg$-modules}, C.R.\
  Acad.\ Sc.\ Paris \textbf{292} (1981), 15--18.

\bibitem{BoeFu97}
B.~Boe and J.~Fu, \emph{Characteristic cycles in hermitian symmetric spaces},
  Canad.~J.~Math. \textbf{49} (1997), no.~3, 417--467.

\bibitem{BorhoBrylinski85}
W.~Borho and J.-L. Brylinski, \emph{Differential operators on homogeneous
  spaces. {III}}, Invent.~Math. \textbf{80} (1985), 1--68.

\bibitem{Carter85}
R.~W. Carter, \emph{Finite {G}roups of {L}ie {T}ype: {C}onjugacy {C}lasses and
  {C}omplex {C}haracters}, Pure and Applied Mathematics, John Wiley \& Sons,
  Inc., New York, 1985.

\bibitem{CollingwoodMcgovern93}
D.~Collingwood and W.~McGovern, \emph{Nilpotent {O}rbits in {S}emisimple {L}ie
  {A}lgebras}, Van Nostrand Reinhold Co., New York, 1993.

\bibitem{EvensMirkovic99}
S.~Evens and I.~Mirkovi\'c, \emph{Characteristic cycles for the loop
  {G}rassmannian and nilpotent orbits}, Duke Math J. \textbf{97} (1999),
  109--126.

\bibitem{FultonPragacz98}
W.~Fulton and P.~Pragacz, \emph{Schubert varieties and degeneracy loci},
  Lecture Notes in Mathematics, vol. 1689, Springer-Verlag, Berlin, 1998.

\bibitem{HechtMilicicSchmidWolf87}
H.~Hecht, D.~Mili{\v{c}}i{\'c}, W.~Schmid, and J.~A. Wolf, \emph{Localization
  and standard modules for real semisimple {L}ie groups. {I}. {T}he duality
  theorem}, Invent. Math. \textbf{90} (1987), no.~2, 297--332.

\bibitem{Kashiwara85}
M.~Kashiwara, \emph{Index theorem for constructible sheaves}, Syst\`emes
  differentiels et singularit\'es (A.~Galligo, M.~Maisonobe, and Ph. Granger,
  eds.), Ast\'erisque, vol. 130, 1985, pp.~193--209.

\bibitem{McGovern00}
W.~McGovern, \emph{A triangularity result for associated varieties of highest
  weight modules}, Communications in Algebra \textbf{28} (2000), no.~4,
  1835--1843.

\bibitem{NishiyamaOchiaiTaniguchi01}
K.~Nishiyama, H.~Ochiai, and K.~Taniguchi, \emph{Bernstein degree and
  associated cycles of {H}arish-{C}handra modules - {H}ermitian case},
  Nilpotent Orbits, Assoicated Cycles and Whittaker Models for Highest Weight
  Representations (K.~Nishiyama, H.~Ochiai, K.~Taniguchi, H.~Yamashita, and
  S.~Kato, eds.), Ast\'erique, vol. 273, Soc. Math. France, 2001, pp.~13--80.

\bibitem{Rossmann91}
W.~Rossmann, \emph{Invariant eigendistributions on a semisimple {L}ie algebra
  and homology classes on the conormal variety. {II}. {R}epresentations of
  {W}eyl groups}, J. Funct. Anal. \textbf{96} (1991), 155--193.

\bibitem{SchmidVilonen00}
W.~Schmid and K.~Vilonen, \emph{Characteristic cycles and wave front cycles of
  representations of reductive {L}ie groups}, Ann. of Math. (2) \textbf{151}
  (2000), no.~3, 1071--1118.

\bibitem{Stanley99}
R.~P. Stanley, \emph{Enumerative {C}ombinatorics. {V}ol. 2}, Cambridge Studies
  in Advanced Mathematics, vol.~62, Cambridge University Press, Cambridge,
  1999.

\bibitem{Trapa07}
P.~Trapa, \emph{Leading-term cycles of {H}arish-{C}handra modules and partial
  orders on components of the {S}pringer fiber}, Compos. Math. \textbf{143}
  (2007), no.~2, 515--540.

\bibitem{Vogan91}
D.~A. Vogan, \emph{Associated varieties and unipotent representations},
  Harmonic Analysis on Reductive Groups (W.~Barker and P.~Sally, eds.),
  Progress in Mathematics, vol. 101, Birkh\"auser, 1991, pp.~315--388.

\end{thebibliography}

\end{document}